\documentclass[letterpaper, oneside]{amsart}

\usepackage{layout}	

\usepackage[
]{geometry}

\usepackage[parfill]{parskip}   
\usepackage{amssymb}
\usepackage{amsthm}
\usepackage{amsmath}
\usepackage{cite}
\usepackage[all]{xy}
\usepackage{enumerate}
\usepackage{verbatim}
\usepackage{mathtools}
\usepackage{hyperref}
\usepackage{colonequals}
\usepackage{mathrsfs}
\usepackage{color}

\theoremstyle{plane}
\newtheorem{thm}{Theorem}[section]
\newtheorem{lemma}[thm]{Lemma}

\newtheorem{prop}[thm]{Proposition}
\newtheorem{cor}[thm]{Corollary}

\theoremstyle{definition}
\newtheorem{defn}[thm]{Definition}
\newtheorem*{rmk}{Remark}
\newtheorem*{ack}{Acknowledgement}

\newtheorem{example}[thm]{Example}

\newcommand{\A}{\mathbb{A}}
\newcommand{\F}{\mathbb{F}}
\newcommand{\Z}{\mathbb{Z}}
\newcommand{\N}{\mathbb{N}}

\newcommand{\G}{\mathbb{G}}

\newcommand{\C}{\mathbb{C}}
\newcommand{\bT}{\mathbb{T}}

\let\k\relax
\newcommand{\k}{\mathbf{k}}

\newcommand{\Y}{\mathcal{Y}}
\newcommand{\B}{\mathcal{B}}
\let\H\relax
\newcommand{\H}{\mathcal{H}}
\newcommand{\cP}{\mathcal{P}}
\newcommand{\cF}{\mathcal{F}}
\newcommand{\cC}{\mathcal{C}}

\newcommand{\g}{\mathfrak{g}}

\newcommand{\fF}{\mathfrak{F}}
\newcommand{\fX}{\mathfrak{X}}

\let\l\relax
\newcommand{\l}{l}
\newcommand{\act}[1]{{}^{#1}}

\newcommand{\br}[1]{\langle{#1}\rangle}

\newcommand{\qlbar}{{\overline{\mathbb{Q}_\ell}}}
\newcommand{\conj}[1]{\underline{#1}}
\newcommand{\charac}[1]{\widehat{#1}}

\DeclareMathOperator{\tr}{\textup{tr}}
\DeclareMathOperator{\ad}{\textup{ad}}
\DeclareMathOperator{\ind}{\textup{ind}}

\DeclareMathOperator{\Spec}{\textup{Spec}}

\DeclareMathOperator{\End}{\textup{End}}
\DeclareMathOperator{\Hom}{\textup{Hom}}

\DeclareMathOperator{\ch}{\textup{char}}
\DeclareMathOperator{\Irr}{\textup{Irr}}

\title{Euler characteristic of analogues of a Deligne-Lusztig variety for $GL_n$}
\author{Dongkwan Kim}
\address{Department of Mathematics\\
  Massachusetts Institute of Technology\\
  Cambridge, MA 02139-4307\\
  U.S.A.}
\email{sylvaner@math.mit.edu}
\date{\today}							

\begin{document}
\begin{abstract} In this paper we give a combinatorial formula to calculate the Euler characteristic of an analogue of a Deligne-Lusztig variety if we replace Frobenius morphism with conjugation by an element for $GL_n$. The main theorem states that it only depends on the unipotent part of the Jordan decomposition of an element and the conjugacy class in the Weyl group. Also it generalizes the formula of the Euler characteristic of a Springer fiber for type A.
\end{abstract}

\maketitle

\renewcommand\contentsname{}
\tableofcontents
\section{Introduction}
Suppose an algebraic reductive group $G$ over an algebraically closed field $\k$ is given. If $\k$ is an algebraic closure of a finite field with some fixed Frobenius morphism, then \cite{dl} defined $X(w)$ for any element $w$ in the Weyl group of $G$ which is now called a Deligne-Lusztig variety. Likewise, if we replace the Frobenius morphism with conjugation by $g\in G$, we obtain its analogue, denoted by $\Y_{w,g}$ in this paper. Note that it can be defined over any field $\k$.

The variety $\Y_{w,g}$ is studied by e.g. \cite{kawanaka},\cite{lu:reflection}, \cite{lu:char1}, \cite{lu:weyltounip},\cite{lu:homogeneity}, \cite{lu:distinguished},\cite{dk:homology}, etc. Also when $w$ is the identity, it coincides with the definition of a Springer fiber of $g \in G$. In this paper we describe the $\ell$-adic Euler characteristic (or Euler-Poincar\'e characteristic) of $\Y_{w,g}$ when $G=GL_n(\k)$. Note that it has a following meaning: we let $\pi_w : Y_w \rightarrow G$ as in the introduction of \cite{lu:char1}. Then the $\ell$-adic Euler characteristic of $\Y_{w,g}$ is the same as that of $(\pi_w)_! \qlbar_{Y_w}$ at $g\in G$, where $\qlbar_{Y_w}$ is a constant $\qlbar$-line bundle on $Y_w$.

The main result of this paper asserts that this Euler characteristic, denoted by $\chi(\Y_{w,g})$, is easy to calculate for $G=GL_n(\k)$. Indeed, it only depends on the unipotent part of $g$ in its Jordan decomposition and the conjugacy class of $w$ in the Weyl group of $G$. Also there is a simple combinatorial formula to calculate such $\chi(\Y_{w,g})$. This generalizes the well-known formula of the Euler characteristic of a Springer fiber for type A. We expect that similar properties hold for reductive groups of other types.

\begin{ack} The author thanks George Lusztig for stimulating discussions and thoughtful comments.
\end{ack}

\section{Some notations and definitions}
Here we fix some notations and definitions which are used throughout this paper. For a group $H$ and subgroup $K \subset H$, we let $N_H(K)$ be the normalizer of $K$ in $H$. 
For any element $h \in H$, we denote the centralizer of $h$ in $H$ by $C_H(h)$. We define $\conj{H}$ to be the set of conjugacy classes in $H$. For any $\cC \in \conj{H}$, we set $C_H(\cC)$ to be $C_H(h)$ for any $h \in \cC$, which is well-defined up to conjugacy. If $H$ is a topological group, then we set $H^0$ to be the identity component of $H$ which is a topological subgroup of $H$. 

For a finite dimensional $\C$-algebra $A$, we denote by $\Irr(A)$ the set of irreducible representations of $A$ over $\C$. If $H$ is a finite group, we often write $\Irr(H)$ instead of $\Irr(\C[H])$. We set $\charac{H}$ to be the set of all virtual characters of $H$ over $\C$, which is equivalent to the $\Z$-span of $\Irr(H)$. For $h \in H$ and $E \in \charac{H}$, we denote by $\tr(h, E)$ the character value of $E$ at $h$. For $\cC \in \conj{H}$ define $\tr(\cC, E)$ to be $\tr(h, H)$ at any $h \in \cC$.

Let $\k$ be an algebraically closed field of characteristic $p$ (which can be zero) and $G=GL_n(\k)$. For a variety $X$ over $\k$ and a prime $\ell \neq p$, $\chi(X)$ denotes the ($\ell$-adic) Euler characteristic of $X$ defined by the following formula.
$$\chi(X) \colonequals \sum_{i\in \Z} (-1)^i \dim_\qlbar H^i_c(X, \qlbar)$$
Recall that $\sum_{i\in \Z} (-1)^i \dim_\qlbar H^i_c(X, \qlbar) = \sum_{i\in \Z} (-1)^i \dim_\qlbar H^i(X, \qlbar)$ by \cite{laumon:euler}. Also we denote the constant $\qlbar$-line bundle on $X$ by $\qlbar_X$.

We fix a standard basis $e_1, \cdots, e_n \in \k^n$ and consider $G$ as the set of invertible $n\times n$ matrices with respect to this fixed basis. Let $B_0\subset G$ be the subgroup of $G$ consisting of all upper triangular invertible matrices which is a Borel subgroup of $G$. Also let $T_0 \subset B_0$ be the subgroup of $B_0$ consisting of all diagonal invertible matrices which is a maximal torus of $G$. Thus $B_0$ stabilizes the standard flag $[0 \subset \br{e_1} \subset \br{e_1, e_2} \subset \cdots \subset \br{e_1, \cdots, e_n} = \k^n]$ and $T_0$ stabilizes each $\br{e_i}$ for $1 \leq i \leq n$.

For any $g \in G$, we let $g=g_sg_u$ be the Jordan decomposition of $g$ such that $g_s \in G$ is semisimple, $g_u$ is unipotent, and $g_sg_u=g_ug_s$. We define the \emph{Jordan type of $g$} to be the partition of $n$ which corresponds to the sizes of Jordan blocks of $g$.

We identify $N_G(T_0)/T_0$ with the Weyl group of $G$ denoted by $W$. Let $S \subset W$ be the set of simple reflections which corresponds to the choice of the Borel subgroup $B_0 \supset T_0$. More specifically, for $1 \leq i \leq n-1$ we set $s_i \in S$ to be the image of the linear map $\k^n \rightarrow \k^n$ in $N_G(T_0)/T_0$ which fixes each $e_j$ for $j\neq i, i+1$ and swaps $e_i$ and $e_{i+1}$. Furthermore, we identify $W = S_n$ where each $s_i$ corresponds to the transposition $(i, i+1)$. Note that $(W, S)$ is a Coxeter group. For $w\in W$, we let $\l(w)$ be the length of $w$ which is the length of any reduced expression with respect to $S$.

For a partition $\lambda =(\lambda_1, \cdots, \lambda_r) \vdash n$, we set $S_\lambda \colonequals S_{\lambda_1} \times \cdots \times S_{\lambda_r}$. We often regard $S_\lambda$ as a subgroup of $S_n$ in a natural way which is well-defined up to conjugacy. Also we let $\lambda! \colonequals |S_\lambda| = \lambda_1!\cdots\lambda_r!$. We define the \emph{cycle type of $w \in S_n$} to be the partition of $n$ corresponding to the sizes of disjoint cycles of $w$. Likewise, we define the cycle type of $\cC \in \conj{W}$ to be that of $w$ for any $w\in \cC$.

For a set $X$ and an endomorphism $\sigma : X \rightarrow X$, we denote by $X^\sigma$ the set of fixed elements in $X$ by $\sigma$. Likewise, if a group $H$ acts on $X$ then we denote by $X^H$ the set of fixed elements in $X$ by all elements in $H$. If $Y$ is a finite set, we denote by $|Y|$ the number of elements in $Y$. Thus in particular, if $\sigma : X \rightarrow X$ has the finite number of fixed elements in $X$, then $|X^\sigma|$ denotes the number of such elements. Also for any $x \in X$, we write $\act{\sigma}x$ for $\sigma(x)$. If $G$ is a group, then we often regard $g$ as an automorphism $\ad(g) :G \rightarrow G$, thus for any $x\in G$ or $X \subset G$ we write $\act{g}x$ for $gxg^{-1}$ and $\act{g}X$ for $gXg^{-1}$. 

Let $\B$ be the flag variety of $G$ which parametrizes all the complete flags in $\k^n$ or equivalently all the Borel subgroups of $G$. We usually identify $\B$ with $G/B_0$. For $B, B' \in \B$, we write $B\sim_w B'$ for some $w \in W$ if there exists $g \in G$ such that $\act{g}B =B_0$ and $\act{g}B' = \act{w}B_0$. For $w \in W$ and $g\in G$, we define the following variety.
$$\Y_{w,g} \colonequals \{ B \in \B \mid B\sim_w \act{g}B\}$$
Note that if we replace $\ad(g)$ by the Frobenius morphism, then this is the Deligne-Lusztig variety $X(w)$ corresponding to $w \in W$ defined by \cite[Definition 1.4]{dl}. The main purpose of this paper is to calculate $\chi(\Y_{w,g})$ for any $w \in W$ and $g \in G$.

\section{Main theorem}
\begin{defn}\label{def:num} For $\lambda, \rho \vdash n$, we define $X_{\rho}^\lambda \colonequals \ind_{S_\lambda}^{S_n} id_{S_\lambda} (w)$ where $w \in S_n$ is any element of cycle type $\rho$.
\end{defn}
It is equivalent to $\br{p_\rho, h_\lambda}$ where $p_\rho$ is the power symmetric function corresponding to $\rho$, $h_\lambda$ is the homogeneous symmetric function corresponding to $\lambda$, and $\br{\ ,\ }$ is the usual scalar product defined on the ring of symmetric functions. Equivalently, if we expand $p_\rho$ in terms of monomial symmetric functions, $X_\rho^\lambda$ corresponds to the coefficient of the monomial corresponding to $\lambda$. For more information we refer readers to \cite[Chapter 7]{stanley}. In particular, $X_\rho^\lambda \in \N$.

Here we present the main theorem of this paper.
\begin{thm}\label{thm:main}
 For $g\in G$, let $g=g_sg_u$ be the Jordan decomposition of $g$ and $g_u \in G$ be of Jordan type $\lambda=(\lambda_1, \cdots, \lambda_r) \vdash n$. Then for any $w\in W$ which is of cycle type $\rho=(\rho_1, \cdots, \rho_s) \vdash n$, we have $\chi(\Y_{w, g}) = X_{\rho}^\lambda$.
\end{thm}
In particular, $\chi(\Y_{w, g})$ only depends on $g_u$ and the conjugacy class of $w$. If $w = id$ so that it is of cycle type $\rho = (1, \cdots, 1)$, we have the following well-known result. (For more information one may refer to \cite{lu:inductionthm}, \cite{fresse}, etc.)
\begin{cor} For $g\in G$ such that its unipotent part $g_u \in G$ is of Jordan type $\lambda=(\lambda_1, \cdots, \lambda_r)\vdash n$, the Euler characteristic of the Springer fiber of $g$ is $\frac{n!}{\lambda_1!\cdots\lambda_r!}$.
\end{cor}

\begin{rmk} Suppose $\k = \overline{\F_q}$ where $q$ is a power of some prime $p> 0$ and $F$ is the (split) Frobenius morphism on $G$ with respect to $\F_q$. Then recall that for a Deligne-Lusztig variety $X(w)$, by \cite[Theorem 7.1]{dl} we have
\[\chi(X(w)) = (-1)^{\l(w)}|G^F|_{p'}|T_w^F|^{-1}\]
where $T_w \subset G$ is the rational torus of type $w\in W$ and $|G^F|_{p'}$ is the largest divisor of $|G^F|$ prime to $p$. (The type of a rational maximal torus will be defined in the next section.) If we regard the expression of $\chi(X(w))$ as a polynomial of $q$, then as $q \rightarrow 1$ we have $X(w) \rightarrow n!\cdot\delta_{w, id}$. By Theorem \ref{thm:main}, this is the same as $\chi(\Y_{w, g})$ when $g \in G$ is semisimple. Indeed, there are several phenomena which assert that ``the Frobenius morphism converges to a (regular) semisimple element as $q\rightarrow 1$", cf. \cite{lu:certain}, \cite[Theorem 4.12]{dk:homology}, etc.
\end{rmk}

\section{Finite field case: Combinatorial method}
In this section we prove Theorem \ref{thm:main} in case of $\k = \overline{\F_q}$ where $q$ is a power of some prime $p>0$. In this case our proof is based on combinatorics. For $\lambda=(\lambda_1, \cdots, \lambda_r), \rho=(\rho_1, \cdots, \rho_s) \vdash n$, we define $P(\rho, \lambda)$ to be the set of functions
$$\zeta:\{1, \cdots, s\} \rightarrow \{1, \cdots, r\}$$
such that for each $1\leq i \leq r$ we have $\lambda_i = \sum_{\zeta(j) = i} \rho_j$. Similarly, we define $[P(\rho, \lambda)]$ to be the set of functions
$$\xi:\{\rho_1, \cdots, \rho_s\} \rightarrow \{\lambda_1, \cdots, \lambda_r\}$$
such that for each $1\leq i \leq r$ we have $\lambda_i = \sum_{\xi(\rho_j) = \lambda_i} \rho_j$. Here we regard $\{\rho_1, \cdots, \rho_s\}$ and $\{\lambda_1, \cdots, \lambda_r\}$ as unordered multisets. Then there is a canonical map
$$[-]: P(\rho, \lambda)\rightarrow [P(\rho, \lambda)]$$
given by ``forgetting the order of parts of the same size." Then it is easy to see that the relation $\zeta\sim \zeta' \Leftrightarrow [\zeta]=[\zeta']$ is an equivalence relation on $P(\rho, \lambda)$, and thus $[P(\rho, \lambda)]$ is identified with the set of equvalence classes of $P(\rho, \lambda)$. 
\begin{example} \label{ex1} Suppose $\lambda = (7, 3), \rho=(3,2,2,2,1) \vdash 10$. Then $P(\rho, \lambda)$ consists of the following functions.
\begin{align*}
\zeta_1: 1,2,3\mapsto 1 \qquad 4, 5 \mapsto 2
\\\zeta_2: 1,2,4\mapsto 1 \qquad 3, 5 \mapsto 2
\\\zeta_3: 1,3,4\mapsto 1 \qquad 2, 5 \mapsto 2
\\\zeta_4: 2,3,4,5\mapsto 1 \qquad 1 \mapsto 2
\end{align*}
Among them $\zeta_1 \sim \zeta_2 \sim \zeta_3$ which are not equivalent to $\zeta_4$. We have
\begin{align*} 
[\zeta_1]=[\zeta_2]=[\zeta_3]&: 3, 2, 2, \mapsto 7, \qquad 2, 1 \mapsto 3
\\ [\zeta_4]&: 2, 2, 2,1 \mapsto 7, \qquad 3 \mapsto 3
\end{align*}
Note that $|P(\rho, \lambda)| = 4 = X_\rho^\lambda$ in this case. This is not a coincidence; We have the following lemma.
\end{example}
\begin{lemma} $X_\rho^\lambda = |P(\rho, \lambda)|$.
\end{lemma}
\begin{proof} We fix a subgroup $S_\lambda \subset S_n$ and $w_\rho \in S_n$ of cycle type $\rho$. By definition $X_\rho^\lambda$ is the number of elements $vS_\lambda \in S_n/S_\lambda$ such that $\act{v^{-1}}w_\rho \in S_\lambda$. But it is easy to show that there is an one-to-one correspondence between such $vS_\lambda$ and elements in $P(\rho, \lambda)$ by looking at the cycle decomposition of $\act{v^{-1}}w_\rho$ in $S_\lambda$.
\end{proof}
\begin{rmk} If we use another definition of $X_\rho^\lambda$ that it is the coefficient of the monomial corresponding to $\lambda$ in the expression of $p_\rho$, then it is a direct consequence of \cite[Proposition 7.7.1]{stanley}.
\end{rmk}

Let $F$ be the Frobenius morphism with respect to $\F_q$. It acts on $G$ by raising each matrix entry to the power $q$. We may assume that $g, g_s, g_u$ are of Jordan normal form, thus $g \in B_0^F$ and $g_s \in T_0^F$. Assume $g_u$ is of Jordan type $\lambda=(\lambda_1, \cdots, \lambda_r)\vdash n$ and $g_s$ is of Jordan type  $\lambda' = (\lambda'_1, \cdots, \lambda'_{r'}) \vdash n$. As each Jordan block of $g_u$ is contained in some Jordan block of $g_s$, there exists a corresponding canonical map 
\begin{equation} \label{eq:jordan}
\phi_g \in [P(\lambda, \lambda')]
\end{equation}
which corresponds to the information of such containment.

We define $L_s := C_G(g_s) = C_G(g_s)^0$, which is isomorphic to $(GL_{\lambda'_1}\times \cdots \times GL_{\lambda'_{r'}})(\k)$ and contains $g_u$ and $T_0$. Also we define $W_s = N_{L_s}(T_0)/T_0$ to be the Weyl group of $L_s$, which we identify with $S_{\lambda'}$. It is naturally a subgroup of $S_n$ by the inclusion map $N_{L_s}(T_0)/T_0 \hookrightarrow N_G(T_0)/T_0$. Note that $(W_s, S\cap W_s)$ is also a Coxeter group.

For any rational maximal torus $T \subset G$, we denote the Deligne-Lusztig induction defined in \cite{dl} by $R^G_T: \charac{T^F} \rightarrow \charac{G^F}$. We set $id_T\in \Irr(T^F)$ to be the trivial character of $T^F$ so that $R^G_T id_T \in \charac{G^F}$ is well-defined. Also for any $g \in G$ such that $\act{g}T_0=T$, the image of $g^{-1}\act{F}g$ in $N_G(T_0)/T_0$ is well-defined up to conjugacy. By abuse of notation, we define the type of $T$ to be one of the following.
\begin{enumerate}[\qquad(a)]
\item The image of $g^{-1}\act{F}g$ in $N_G(T_0)/T_0$, denoted by $w\in W$.
\item $\rho \vdash n$ where $\rho$ is the cycle type of $w$.
\item $\cC \in \conj{W}$ for $w\in \cC$.
\end{enumerate}
Conversely, for such $w, \rho, \cC$ we denote by $T_w, T_\rho,$ or $T_{\cC}$ a rational maximal torus of type $w, \rho,$ or $\cC$, respectively. It is well-defined up to conjugation by $G^F$.

We want to calculate $|(\Y_{w,g})^F|$ and show that it is a ``polynomial of $q$." This is related to $\chi(\Y_{w,g})$ by the following lemma.

\begin{lemma} \label{lem:char} Suppose a variety $X$ is defined over $\F_q$. If there exists a polynomial $\phi(x) \in \C[x]$ such that $|X^{F^m}| = \phi(q^m)$ for all but finitely many $m \in \Z_{>0}$, then $\chi(X) = \phi(1)$.
\end{lemma}
\begin{proof} For $i \in \Z$, let $\alpha_{i,1}, \cdots, \alpha_{i, d_i}$ be eigenvalues of Frobenius on $H^i_c(X, \qlbar)$. Then by Lefschetz trace formula we have
$$|X^{F^m}| = \sum_{i \in \Z} (-1)^i \sum_{j=1}^{d_i} \alpha_{i, j}^m$$
for any $m \in \Z_{>0}$. Note that $\chi(X) = \sum_{i \in \Z} (-1)^i \sum_{j=1}^{d_i} \alpha_{i, j}^0 = \sum_{i \in \Z} (-1)^i d_i$. By canceling out $\alpha_{i,j}$ with the same value and different parity of $i$, we may assume that
$$|X^{F^m}| = \alpha_1^m + \cdots + \alpha_a^m - \beta_1^m - \cdots - \beta_b^m$$
for some algebraic integers $\alpha_1,\cdots, \alpha_a, \beta_1, \cdots, \beta_b$ such that $\alpha_i$ are different from $\beta_j$. Note that $\chi(X) = a-b$. By assumption, there exists $\phi(x) = c_d x^d +\cdots +c_1 x+c_0\in \C[x]$ such that
$$\phi(q^m)=c_d q^{md} +\cdots +c_1 q^m+c_0=\alpha_1^m + \cdots + \alpha_a^m - \beta_1^m - \cdots - \beta_b^m$$
for all but finitely many $m \in \Z_{>0}$. But then it is easy to see that it is true for all $m \in \Z$. Now the result follows by letting $m=0$.
\end{proof}

In order to calculate $|(\Y_{w,g})^F|$ we follow the argument in \cite[Proposition 2.1]{lu:reflection} or \cite[1.2]{lu:weyltounip}. We first consider the $G^F$-module $\fF = \Hom_\C(\B^F, \C)$ where the action is defined by $(g\cdot f)(B) \colonequals f(\act{g^{-1}}B)$ for $g\in G^F$ and $f \in \fF$. Then the Iwahori-Hecke algebra $\H_q \colonequals \End_{G^F}(\fF)$ over $\C$ is defined and it has a $\C$-basis $\bT_w \in \H_q$ labeled by elements in $W$ which acts on $\fF$ as follows.
$$\bT_w(f)(B) = \sum_{B' \in \B^F, B \sim_w B'} f(B')$$
As an algebra $\H_q$ is generated by $\bT_{s_i}$ for $s_i \in S$ with relations given by
\begin{enumerate}
\item $\bT_{s_i} \bT_{s_{i+1}} \bT_{s_i} = \bT_{s_{i+1}}\bT_{s_i}\bT_{s_{i+1}}$ for $1 \leq i \leq n-2,$
\item $\bT_{s_i} \bT_{s_{j}} = \bT_{s_{j}}\bT_{s_i}$ if $|i-j|>1$,
\item $\bT_{s_i}^2 = (q-1)\bT_{s_i}+q\bT_{id}$ for $1 \leq i \leq n-1$.
\end{enumerate}
Also for any $E \in \Irr(W)$, there exists $E_q \in \Irr(\H_q)$ with the following property. If we regard $q$ in $\H_q$ as an indeterminate, then $\H_q$ becomes $\C[W]$ as $q\rightarrow1$ with $\bT_w \rightarrow w$, and $\tr(\bT_w, E_q)$ becomes $\tr(w, E)$. Now we have the following decomposition
$$\fF = \bigoplus_{E \in \Irr(W)} E_q \times \fX_E$$
of $\H_q \times G^F$-modules. In case of $G=GL_n(\k)$, $\fX_E$ is given by the following formula. (e.g. \cite[Proposition 12.6]{lu:orangebook})
$$\fX_E = \frac{1}{|W|}\sum_{w \in W} \tr(w, E) R^G_{T_w} id_{T_w} = \sum_{\cC \in W} \frac{1}{|C_W(\cC)|} \tr(\cC, E) R^G_{T_\cC} id_{T_\cC} $$

%

For $B \in \B^F$, we let $1_{B}$ be the indicator function, i.e. $1_{B}(B') = \delta_{B, B'}$. Then we have
$$(g\bT_w 1_{B})(B') = \sum_{B'' \in \B^F, B' \sim_w B''} 1_{B}(\act{g^{-1}}B'')$$
which is 1 if $B' \sim_w \act{g}B$ and 0 otherwise. In other words, we have
$$g\bT_w1_{B}=\sum_{B' \in \B^F, B' \sim_{w} \act{g}B} 1_{B'}.$$
Thus it follows that $|(\Y_{w,g})^F| = \tr(g\bT_w, \fF).$ As $\fF = \bigoplus_{E \in \Irr(W)} E_q \times \fX_E$, we have
\begin{align}
|(\Y_{w,g})^F| &= \sum_{E \in \Irr(W)} \tr(\bT_w, E_q) \tr(g, \fX_E) \nonumber
\\&= \sum_{E \in \Irr(W)}\sum_{\cC \in \conj{W}} \frac{1}{|C_W(\cC)|} \tr(\bT_w, E_q)\tr(\cC, E)R^G_{T_{\cC}}id_{T_\cC}(g). \label{eq:poly1}
\end{align}

We use the character formula for Deligne-Lusztig characters. (One may refer to \cite{dl} or \cite{dm:book}.) For a rational maximal torus $T \subset G$ we have
\begin{align*}
R^G_{T} id_T (g) &= \frac{1}{|L_s^F|}\sum_{h \in G^F, \act{h}T \subset L_s}R_{\act{h}T}^{L_s} id_{\act{h}{T}}(g_u)
\\&=\frac{1}{|L_s^F|}\sum_{T'\subset L_s}|N_G(T)^F|R_{T'}^{L_s} id_{T'}(g_u)
\end{align*}
where $T' \subset L_s$ runs over all the rational maximal tori in $L_s$ which is conjugate to $T$ by $G^F$.
Therefore for any $\cC \in \conj{W}$ we have
\begin{align}
R^G_{T_\cC} id_{T_\cC} (g) &=\sum_{\tilde{\cC} \in \conj{W_s}, \tilde{\cC}\subset \cC} \frac{|N_G(T_{\cC})^F|}{|N_{L_s}(T_{\tilde{\cC}})^F|} R_{T_{\tilde{\cC}}}^{L_s} id_{T_{\tilde{\cC}}}(g_u) \nonumber
\\&=\sum_{\tilde{\cC} \in \conj{W_s}, \tilde{\cC}\subset \cC} \frac{|C_W(\cC)|}{|C_{W_s}(\tilde{\cC})|} R_{T_{\tilde{\cC}}}^{L_s} id_{T_{\tilde{\cC}}}(g_u) \label{eq:dlchar}
\end{align}
where $T_{\tilde{\cC}} \subset L_s$ is the rational maximal torus of $L_s$ of type $\tilde{\cC} \in W_s$. Here we use the fact that $g_s \in T_0^F$, thus $F$ acts trivially on both $W$ and $W_s$. Thus we have
$$|(\Y_{w,g})^F| = \sum_{E \in \Irr(W)}\sum_{\tilde{\cC} \in \conj{W_s}} \frac{1}{|C_{W_s}(\tilde{\cC})|} \tr(\bT_w, E_q)\tr(\tilde{\cC}, E|_{W_s})R_{T_{\tilde{\cC}}}^{L_s} id_{T_{\tilde{\cC}}}(g_u).
$$
Here we have an interesting by-product.
\begin{prop} If $g,g' \in G^F$ are of Jordan normal form and satisfy $g_u = g'_u$ and $C_G(g_s) = C_G(g'_s)$, then $|(\Y_{w,g})^F| = |(\Y_{w, g'})^F|$.
\end{prop}
\begin{proof} Note that the formula above only depends on $w \in W$, $g_u$, and $L_s = C_G(g_s)$.\end{proof}


We recall some facts about Green polynomials and its relation to Deligne-Lusztig characters. It is known that for any unipotent $u \in G^F$ of Jordan type $\lambda \vdash n$ and for any $\rho \vdash n$, we have
\[
R^G_{T_\rho} id_{T_\rho} (u) = Q_{\rho}^{\lambda}(q)
\]
where $Q_{\rho}^{\lambda}$ is the Green polynomial defined in \cite{green}. By \cite[Chapter III.7]{macdonald} it is easy to check that
\begin{equation} \label{eq:green}
\left(R^G_{T_\rho} id_{T_\rho} (u)\right)_{q=1}=Q^\lambda_{\rho}(1) = X_{\rho}^\lambda
\end{equation}
where $X_{\rho}^\lambda$ is as in Definition \ref{def:num}.

 As $\tr(\bT_w, E_q)$ is a polynomial of $q$ and $R_{T_{\tilde{\cC}}}^{L_s} id_{T_{\tilde{\cC}}}(g_u)$ is a product of polynomials of $q$, we see that $|(\Y_{w,g})^F|$ is also a polynomial of $q$. (Recall that $L_s \simeq GL_{\lambda'_1} \times \cdots \times GL_{\lambda'_{r'}}(\k)$, thus $R_{T_{\tilde{\cC}}}^{L_s} id_{T_{\tilde{\cC}}}(g_u)$ is the product of Green polynomials which correspond to each $GL_{\lambda'_i}$ for $1\leq i \leq r'$.) Thus by Lemma \ref{lem:char}, $\chi(\Y_{w,g}) = |(\Y_{w,g})^F|_{q=1}$. Using (\ref{eq:poly1}) we have
\begin{align*}
|(\Y_{w,g})^F|_{q=1} &= \sum_{E \in \Irr(W)}\sum_{\cC \in \conj{W}} \frac{1}{|C_W(\cC)|} \tr(w, E)\tr(\cC, E)\left(R^G_{T_{\cC}}id_{T_\cC}(g)\right)_{q=1}= \left(R^G_{T_{w}}id_{T_w}(g)\right)_{q=1}.
\end{align*}
Here we use the ``column orthogonality" of the character table, i.e. for any $\cC, \cC' \in \conj{W}$ we have
\[\sum_{E \in \Irr(W)} \tr(\cC, E) \tr(\cC', E) = \delta_{\cC, \cC'} \cdot |C_W(\cC)|.\]
Using (\ref{eq:dlchar}) we have
\[ \left(R^G_{T_{w}}id_{T_w}(g)\right)_{q=1} = \sum_{\tilde{\cC} \in \conj{W_s}, \tilde{\cC}\subset \cC_w} \frac{|C_W(\cC_w)|}{|C_{W_s}(\tilde{\cC})|} \left(R_{T_{\tilde{\cC}}}^{L_s} id_{T_{\tilde{\cC}}}(g_u)\right)_{q=1}\]
where $\cC_w \in \conj{W}$ is the conjugacy class containing $w \in W$.

Suppose $\rho \vdash n$ is the cycle type of $w$ and recall that $g_u, g_s$ are of Jordan type $\lambda=(\lambda_1, \cdots, \lambda_{r}), \lambda'=(\lambda'_1, \cdots, \lambda'_{r'}) \vdash n$, respectively. Then there is a bijection 
\[\{\tilde{\cC} \in \conj{W_s} \mid \tilde{\cC} \subset \cC_w\} \rightarrow [P(\rho, \lambda')] : \tilde{\cC} \mapsto \psi_{\tilde{\cC}}\]
which corresponds to the cycle type of $\tilde{\cC}$ in $W_s\simeq S_{\lambda'_1} \times \cdots \times S_{\lambda'_{r'}}$. Then using (\ref{eq:green}) it is easy to show that
\[\frac{|C_W(\cC_w)|}{|C_{W_s}(\tilde{\cC})|}\left(R_{T_{\tilde{\cC}}}^{L_s} id_{T_{\tilde{\cC}}}(g_u)\right)_{q=1} = |\{\zeta \in P(\rho, \lambda) \mid \phi_g\circ[\zeta] = \psi_{\tilde{\cC}}\}|.\]
where $\phi_g \in [P(\lambda, \lambda')]$ is the same as (\ref{eq:jordan}).

\begin{example} Suppose $\rho = (3,2,2,2,1), \lambda=(7,3) \vdash 10, \zeta_1, \cdots, \zeta_4 \in P(\rho, \lambda)$ as in Example \ref{ex1}. First we consider $\lambda'=(10)$ and $\phi_g : 7, 3 \mapsto 10$. Then there is only one possible choice of $\psi_{\tilde{\cC}}$, namely $\psi_{\tilde{\cC}} : 3,2,2,2,1 \mapsto 10$. Therefore, $\zeta_1, \cdots, \zeta_4 \in P(\rho, \lambda)$ are all contained in $\{\zeta \in P(\rho, \lambda) \mid \phi_g\circ[\zeta] = \psi_{\tilde{\cC}}\}$. On the other hand, in that case $\frac{|C_W(\cC_w)|}{|C_{W_s}(\tilde{\cC})|} = 1$ and $\left(R_{T_{\tilde{\cC}}}^{L_s} id_{T_{\tilde{\cC}}}(g_u)\right)_{q=1}=4$.

On the other hand, this time we suppose $\lambda'=(7, 3) \vdash 10$ and $\phi_g: 7 \mapsto 7, 3 \mapsto 3$. Then there are two choice of $\psi_{\tilde{\cC}}$, namely, 
\begin{align*}
\psi_1 &: 3, 2, 2 \mapsto 7, \qquad 2,1 \mapsto 3
\\\psi_2 &: 2, 2, 2,1 \mapsto 7, \qquad 3 \mapsto 3.
\end{align*}
Now $\zeta_1, \zeta_2, \zeta_3$ is contained in $\{\zeta \in P(\rho, \lambda) \mid \phi_g\circ[\zeta] = \psi_{1}\}$ and $\zeta_4$ is contained in $\{\zeta \in P(\rho, \lambda) \mid \phi_g\circ[\zeta] = \psi_{2}\}$. Furthermore, in the first case we have $\frac{|C_W(\cC_w)|}{|C_{W_s}(\tilde{\cC})|} = \frac{3!}{2!1!} = 3$ and $\left(R_{T_{\tilde{\cC}}}^{L_s} id_{T_{\tilde{\cC}}}(g_u)\right)_{q=1}=1$, whereas in the second case $\frac{|C_W(\cC_w)|}{|C_{W_s}(\tilde{\cC})|} = \frac{3!}{3!} = 1$ and $\left(R_{T_{\tilde{\cC}}}^{L_s} id_{T_{\tilde{\cC}}}(g_u)\right)_{q=1}=1$.
\end{example}

It is clear that for any $\zeta \in P(\rho, \lambda)$ there exists a unique $\tilde{\cC} \in \conj{W_s}$ which satisfies $\phi_g\circ[\zeta] = \psi_{\tilde{\cC}}$. Thus
\begin{align*}
\left(R^G_{T_{w}}id_{T_w}(g)\right)_{q=1} &= \sum_{\tilde{\cC} \in \conj{W_s}, \tilde{\cC}\subset \cC_w} |\{\zeta \in P(\rho, \lambda) \mid \phi_g\circ[\zeta] = \psi_{\tilde{\cC}}\}|
\\&=|P(\rho, \lambda)| = X_{\rho}^\lambda.
\end{align*}
Thus Theorem \ref{thm:main} is proved for $\k = \overline{\F_q}$.

\section{Characteristic $p$: Spread-out}
In this section we assume that $\ch \k =p>0$. Here we use spread-out technique which is as follows. We may assume that there exists a subalgebra $A \subset \k$ which satisfies the following property.
\begin{enumerate}
\item $A$ is finitely generated over $\F_p$, i.e. $A = \F_p[a_1, \cdots, a_d] \subset \k$.
\item $g \in G=GL_n(\k)$ is defined over $A$, i.e. it is contained in the image of $GL_n(A) \rightarrow GL_n(\k)$.
\end{enumerate}
Thus we may write $g \in GL_n(A)$. We can also assume that $g$ is of Jordan normal form in $GL_n(A)$. Then there exists $\tilde{\Y}$ defined over $A$ with the following fiber diagram.
\[
\xymatrix{\tilde{\Y} \ar[d]^{f} & \Y_{w,g} \ar[l]_{\tilde{\iota}}\ar[d]^{\tilde{f}} \\ \Spec A & \Spec \k \ar[l]_{\iota}}
\]
By proper base change, we have $\iota^*f_!\qlbar_{\tilde{\Y}} \simeq \tilde{f}_!\tilde{\iota}^*\qlbar_{\tilde{\Y}} = \tilde{f}_!\qlbar_{\Y_{w,g}}$. Thus
\[\chi(\Y_{w,g}) = \sum_{i\in \Z} (-1)^i \dim H^i_c(\Y_{w,g}, \qlbar) = \sum_{i\in \Z} (-1)^i \dim R^i\tilde{f}_!\qlbar_{\Y_{w,g}}= \sum_{i\in \Z} (-1)^i \dim (R^if_!\qlbar_{\tilde{\Y}})_\eta\]
where $\eta \in \Spec A$ is the image of $\iota : \Spec \k \rightarrow \Spec A$ or the generic point of $\Spec A$.
Let $x \in \Spec A$ be a closed point in $\Spec A$. Then its residue field $k(x)$ is a finite extension over $\F_p$, thus is a finite field. Also we have a following fiber diagram.
\[
\xymatrix{\tilde{\Y}_x \ar[r]^{\tilde{\varsigma}} \ar[d]^{\hat{f}}&\tilde{\Y} \ar[d]^{f} & \Y_{w,g} \ar[l]_{\tilde{\iota}}\ar[d]^{\tilde{f}} \\ \Spec k(x)\ar[r]^\varsigma& \Spec A & \Spec \k \ar[l]_{\iota}}
\]
We consider the image of $g = g_s g_u \in GL_n(A)$ in $GL_n(k(x))$, say $\bar{g} = \bar{g}_s\bar{g}_u \in GL_n(k(x))$. Then this again gives the Jordan decomposition of $\bar{g}$ and $\tilde{\Y}_{x}$ is nothing but $\Y_{w, \bar{g}}$ defined over $k(x)$ with respect to $\bar{g} \in GL_n(k(x))$ and $w\in W$.

By the result in the previous section, we know that $\chi(\tilde{\Y}_x) = X_\rho^\lambda$ where $w$ is of cycle type $\rho$ and $\bar{g}_u$ is of Jordan type $\lambda$. As we assume that $g = g_sg_u$ is Jordan normal form, the Jordan type of $g_u$ and $\bar{g}_u$ are the same, thus $g_u$ is also of Jordan type $\lambda$. Again by proper base change, we have $\varsigma^*f_!\qlbar_{\tilde{\Y}} \simeq \hat{f}_!\tilde{\varsigma}^*\qlbar_{\tilde{\Y}} = \hat{f}_!\qlbar_{\tilde{\Y}_x}$. Thus in particular the Euler characteristic of $(f_!\qlbar_{\tilde{\Y}})_x$ is equal to $X_\rho^\lambda$. Since the choice of $x$ is arbitrary, it is true for any closed point in $\Spec A$.

However, $f_!\qlbar_{\tilde{\Y}}$ is a constructible complex of $\ell$-adic sheaves on $\Spec A$. Therefore, the Euler characteristic of $f_!\qlbar_{\tilde{\Y}}$ at the generic point, or $\chi(\Y_{w,g})$, is generically the same as that at closed points, or $X_\rho^\lambda$. Thus Theorem \ref{thm:main} is proved for any $\k$ of characteristic $\neq 0$.

\section{Characteristic 0: Geometric method}
In this section we suppose $\ch \k =0$. Here we use totally different method from above which is based on geometry of $\Y_{w,g}$. First we need a lemma.

\begin{lemma} \label{lem:conj} Suppose $w \in W$ and $s \in S$. Then for any $g \in G$, we have $\chi(\Y_{w, g}) = \chi(\Y_{sws,g})$.
\end{lemma}
\begin{proof} We follow the argument in \cite[Theorem 1.6]{dl}. It suffices to show the following.
\begin{enumerate}[\quad(a)]
\item For $u,v \in W$ such that $\l(u)+\l(v) = \l(uv) = \l(vu)$, $\chi(\Y_{uv, g}) = \chi(\Y_{vu, g})$.
\item For $w \in W$ and $s \in S$ such that $\l(sws) = \l(w) - 2$, $\chi(\Y_{w, g}) = \chi(\Y_{sws, g})$.
\end{enumerate}
In case (a), there is an isomorphism
\[\Y_{uv, g} \simeq \Y_{vu, g} : B \mapsto B'\]
such that $B \sim_u B' \sim_{v} \act{g}B.$ In case (b), there is a stratification of $\Y_{w, g}$ into two pieces, one of which maps to $\Y_{sw,g}$ with each fiber isomorphic to $\A^1 -\{0\}$ and another one maps to $\Y_{sws, g}$ with fiber $\A^1$. Since $\chi(\A^1 -\{0\})=0$ and $\chi(\A^1) = 1$, we use the lemma below to conclude the result.
\end{proof}

\begin{lemma} \label{lem:fib}Let $f:Y\rightarrow X$ be a morphism of varieties over $\k$. Then there is a finite stratification $X = \sqcup_{m \in \Z}X_m$, i.e. $X_m\neq \emptyset$ for finitely many $m\in \Z$, such that for any $x \in X_m$, $f^{-1}(x)$ has Euler characteristic equal to $m$. Furthermore, in this case $\chi(Y) = \sum_{m\in \Z}m\cdot\chi(X_m)$. In particular, if $\chi(f^{-1}(x))$ is constant for any $x \in X$ then $\chi(Y) = \chi(f^{-1}(x))\chi(X)$ for any $x\in X$.
\end{lemma}
\begin{proof} We have $\chi(Y) = \chi(X, f_!\qlbar_Y)$ where the latter is the Euler characteristic of the complex $f_!\qlbar_Y$ on $X$. Note that $f_!\qlbar_Y$ is a constructible complex of $\ell$-adic sheaves on $X$. Now it follows from the comparison theorem between \'etale and complex topology, and an analogous statement in complex topology.
\end{proof}
\begin{rmk} This is in general not true in characteristic $p\neq 0$. Suppose the Artin-Schreier covering $\A^1 \rightarrow \A^1: x \mapsto x^p -x$. Then every fiber consists of $p$ points, but $1=\chi(\A^1) \neq p\cdot\chi(\A^1)=p$.
\end{rmk}
Thus it suffices to prove Theorem \ref{thm:main} when $w\in W$ has the minimal length in its conjugacy class, which we assume from now on. If $w$ has a full support, i.e. if $w$ is a Coxeter element, we have a simple answer.
\begin{lemma} \label{lem:cox}Suppose $w\in W$ is a Coxeter element. Then $\chi(\Y_{w, g})=1$ if $g_u$ is regular unipotent, and otherwise zero.
\end{lemma}
\begin{proof} It is known that $\Y_{w, g} \neq \emptyset$ only if $g$ is regular. Thus we assume $g$ is regular. By \cite[5.8]{lu:weyltounip} $\Y_{w, g}$ consists of finitely many $C_G(g)/Z_G$-orbits with finite isotropy groups, where $Z_G$ is the center of $G$. If $g_u$ is not regular unipotent, then $C_G(g)$ contains a noncentral torus, thus it follows that $\chi(\Y_{w, g})=0$ as $\chi(\A^1-\{0\})=0$. If $g_u$ is regular unipotent then $g_s$ is central, thus $\Y_{w, g}=\Y_{w, g_u}$. Also by \cite[0.3]{lu:homogeneity} $\Y_{w, g_u}$ is an affine space since so is $C_G(g_u)/Z_G$. It follows that $\chi(\Y_{w,g})=1$.
\end{proof}

Since any Coxeter element is of cycle type $\rho=(n)\vdash n$ and $X_{(n)}^\lambda=\delta_{(n),\lambda}$, Theorem \ref{thm:main} is proved in this case. Thus we assume $w\in W$ is minimal but not a Coxeter element. Again by Lemma \ref{lem:conj}, we may assume that $w$ has the following cycle decomposition.
\[w = (1, 2, \cdots, \rho_1)(\rho_1+1, \cdots, \rho_1+\rho_2) \cdots (\rho_1+\cdots+\rho_{s-1}+1, \cdots, n)\]
(Recall that $\rho=(\rho_1, \cdots, \rho_s)\vdash n$ is the cycle type of $w\in W$.) The condition that $w$ is not a Coxeter element implies that $1\leq \rho_1 < n$.

For any flag $\cF = [0=V_0 \subset V_1 \subset \cdots \subset V_{n-1} \subset V_n=\k^n] \in \B$ such that $\dim V_i = i$ for $0\leq i \leq n$, if $\cF \in \Y_{w,g}$ then $\act{g}V_{\rho_1}=V_{\rho_1}$ by the definition of $\Y_{w,g}$. If we let $P_0 \subset G$ be the standard parabolic subgroup containing $B_0$ which corresponds to $\{s_1, \cdots, s_{\rho_1 -1}, s_{\rho_1+1}, \cdots, s_{n-1}\} \subset S$, then $G/P_0 \simeq Gr(\rho_1, \k^n)$ parametrizes $\rho_1$-dimensional subspaces of $\k^n$, which we denote by $\cP$. We have a canonical morphism $\pi : \B \rightarrow \cP$ which corresponds to $G/B_0 \rightarrow G/P_0$. Then the condition above implies that $\pi(\Y_{w,g}) \subset \cP^g$.

%
%
Now we suppose $g=g_sg_u\in G$ is of Jordan normal form and let $g=g_s+g_n$ be the Jordan decomposition as an element in $\textup{Lie}(G) = \mathfrak{gl}_n(\k)$, so that $g_u = g_n+id_{\k^n}$. Also we consider the following 1-parameter subgroup
\[ \gamma : \G_m(\k) \rightarrow G : t \mapsto \textup{diag}(1, t, \cdots, t^{n-1})\]
Where $\G_m$ is an one-dimensional torus. Then $\act{\gamma(t)}g = \act{\gamma(t)}g_s +\act{\gamma(t)}g_n = g_s + tg_n$ is again the Jordan decomposition of $\act{\gamma(t)}g$. Also, for any subspace $V \subset \k^n$, it is easy to see that $g$ stabilizes $V$ if and only if $\act{\gamma(t)}g$ stabilizes $V$. As a result, $\gamma$ defines an action on $\cP^g$ by $t\cdot V \colonequals \act{\gamma(t)^{-1}}V$. 

We wish to apply Lemma \ref{lem:fib} to the restriction of $\pi$ on $\Y_{w, g}$ which we again denote by $\pi$. We claim that every fiber of $\pi: \Y_{w,g} \rightarrow \cP^g$ at each point on a fixed $\G_m(\k)$-orbit on $\cP^g$ is isomorphic to one another. Indeed, suppose $V \in \cP^g$ and let
\[ w_1 = (1, 2, \cdots, \rho_1), \quad w_2 = (\rho_1+1, \cdots, \rho_1+\rho_2) \cdots (\rho_1+\cdots+\rho_{s-1}+1, \cdots, n)\]
such that $w=w_1w_2 = w_2w_1$. Then the fiber of $\pi$ at $V\in \cP^g$ is isomorphic to $\Y_{w_1', g|_{V}} \times \Y_{w_2', g|_{V'}}$ where $V' \colonequals \k^n/V$ and
\begin{align*}
w_1' &= (1, \cdots, \rho_1) \in S_{\rho_1}, 
\\w_2' &= (1, \cdots,\rho_2)(\rho_2+1, \cdots, \rho_2+\rho_3)\cdots (\rho_2+\cdots+\rho_{s-1}+1, \cdots, n-\rho_1)\in S_{n-\rho_1}.
\end{align*}
Likewise on $\act{\gamma(t)^{-1}}V$ the fiber is isomorphic to $\Y_{w_1', g|_{\act{\gamma(t)^{-1}}V}} \times \Y_{w_2', g|_{\act{\gamma(t)^{-1}}V'}}$ where $\act{\gamma(t)^{-1}}V' \colonequals \k^n/\act{\gamma(t)^{-1}}V$. This is again isomorphic to $\Y_{w_1', \act{\gamma(t)}g|_{V}} \times \Y_{w_2', \act{\gamma(t)}g|_{V'}} = \Y_{w_1', (g_s+tg_n)|_{V}} \times \Y_{w_2', (g_s+tg_n)|_{V'}}$ by conjugation.

Note that $g|_V = g_s|_V + g_n|_V$ and $(g_s+tg_n)|_{V} = g_s|_V+tg_n|_V$ are again the Jordan decompositions of $g$ and $(g_s+tg_n)|_{V}$, respectively. Then it is easy to see that $g|_V$ and $(g_s+tg_n)|_{V}$ are conjugate under $GL(V)$. Likewise, $g|_{V'}$ and $(g_s+tg_n)|_{V'}$ are conjugate under $GL(V')$. But this means that $\Y_{w_1', (g_s+tg_n)|_{V}} \times \Y_{w_2', (g_s+tg_n)|_{V'}}$ is isomorphic to $\Y_{w_1', g|_{V}} \times \Y_{w_2', g|_{V'}}$, which is what we want.

Now for $\pi: \Y_{w,g} \rightarrow \cP$, we set $\cP =\sqcup_{m\in \Z}X_m$ as in Lemma \ref{lem:fib}. (Note that $\cP - \cP^g \subset X_0$.) Then the claim above asserts that every $X_m$ is preserved by $\G_m(\k)$-action. Also note that this action on $\cP$ has finite number of fixed points, namely,
\[\cP^{\G_m(\k)}=\{\br{e_{j_1}, \cdots, e_{j_{\rho_1}}} \subset \k^n \mid \{j_1, \cdots, j_{\rho_1}\} \subset \{1, \cdots, n\}\}.\]
Thus by Lemma \ref{lem:fib} and the theorem of Bialynicki-Birula \cite{b-b} we conclude that 
\[\chi(\Y_{w, g}) = \sum_{V \in \cP^{\G_m(\k)}} \chi(\pi^{-1}(V)) =  \sum_{V \in \cP^{\G_m(\k)}, \act{g}V= V}  \chi(\Y_{w_1', g|_{V}}) \cdot \chi(\Y_{w_2', g|_{V'}}) \]
where again $V' = \k^n/V$.

We use Lemma \ref{lem:cox}. Since $w_1' \in S_{\rho_1}$ is a Coxeter element, for any $V \in \cP^{\G_m(\k)}$ we have $\chi(\Y_{w_1', g|_{V}}) =1$ if and only if $g$ stabilizes $V$ and $g_u|_V$ is regular unipotent. Otherwise either $\pi^{-1}(V)$ is empty or $\chi(\Y_{w_1', g|_{V}})=0$, thus $V \in X_0$. Therefore if we let $\mathfrak{V} \colonequals \{ V \in \cP^{\G_m(\k)} \mid \act{g}V = V, g_u|_V \textup{ is regular unipotent}\}$ then we have
\[ \chi(\Y_{w,g}) = \sum_{V \in \mathfrak{V}} \chi(\Y_{w_2', g|_{V'}}).\]
There is an one-to-one correspondence between such $V \in \mathfrak{V}$ and Jordan blocks of $\g_u$ of size $\geq \rho_1$. Then the Jordan type of $g|_{V'}$ is a partition of $n-\rho_1$ obtained by subtracting $\rho_1$ from the part of $\lambda = (\lambda_1, \cdots, \lambda_r) \vdash n$ which corresponds to such Jordan block. We set $\rho'$ to be the cycle type of $w_2'$, which is the same as the partition of $n-\rho_1$ obtained by removing $\rho_1$ from $\rho$. Then by induction on $n$ it is easy to see that
\begin{align*}
 \chi(\Y_{w,g}) &= \sum_{V \in \mathfrak{V}}  \chi(\Y_{w_2',g|_{V'}})=\sum_{i=1}^r |\{ \zeta \in P(\rho, \lambda) \mid \zeta(1) = i\}| = |P(\rho, \lambda)| =X_\rho^\lambda.
 \end{align*}
But this is exactly the statement of Theorem \ref{thm:main}.

%
%
%

\bibliographystyle{amsalphacopy}
\bibliography{euler}

\providecommand{\bysame}{\leavevmode\hbox to3em{\hrulefill}\thinspace}
\providecommand{\MR}{\relax\ifhmode\unskip\space\fi MR }
\providecommand{\MRhref}[2]{%
  \href{http://www.ams.org/mathscinet-getitem?mr=#1}{#2}
}
\providecommand{\href}[2]{#2}
\begin{thebibliography}{Kaw75}

\bibitem[Bia73]{b-b}
Bialynicki{-}{B}irula, A., \emph{On fixed point schemes of actions of
  multiplicative and additive groups}, Topology \textbf{12} (1973), no.~1,
  99--103.

\bibitem[DL76]{dl}
Deligne, P. and Lusztig, G., \emph{Representations of reductive groups over a
  finite field}, Ann. Math \textbf{103} (1976), 103--161.

\bibitem[DM91]{dm:book}
Digne, F. and Michel, J., \emph{Representations of finite groups of {L}ie
  type}, London {M}athematical {S}ociety {S}tudent {T}exts, vol.~21, Cambridge
  {U}niversity {P}ress, 1991.

\bibitem[Fre09]{fresse}
Fresse, L., \emph{Betti numbers of {S}pringer fibers in type {A}}, Journal of
  {A}lgebra \textbf{322} (2009), 2566 -- 2579.

\bibitem[Gre55]{green}
Green, J.~A., \emph{The characters of the finite general linear groups},
  Transactions of the {A}merican {M}athematical {S}ociety (1955), no.~2,
  402--447.

\bibitem[Kaw75]{kawanaka}
Kawanaka, N., \emph{Unipotent elements and characters of finite chevalley
  groups}, Osaka J. Math \textbf{12} (1975), 523--554.

\bibitem[Kim16]{dk:homology}
Kim, D., \emph{Homology class of a {D}eligne-{L}usztig variety and its
  analogues}, Available at \url{http://arxiv.org/abs/1603.09295}, 2016,
  ar{X}iv:1603.09295 [math.{R}{T}].

\bibitem[Lau81]{laumon:euler}
Laumon, G., \emph{Comparaison de caract{\'e}ristiques
  d'{E}uler{-}{P}oincar{\'e} en cohomologie $l$-adique.}, Comptes {R}endus des
  S{\'e}ances de l'{A}cad{\'e}mie des {S}ciences. {S}{\'e}rie {I}.
  {M}ath{\'e}matique \textbf{292} (1981), no.~3, 209--212.

\bibitem[Lus80]{lu:reflection}
Lusztig, G., \emph{On the reflection representation of a finite {C}hevalley
  group}, Representation {T}heory of {L}ie {G}roups, London {M}athematical
  {S}ociety lecture note series, vol.~34, Cambridge {U}niversity {P}ress, 1980,
  pp.~325--337.

\bibitem[Lus84]{lu:orangebook}
\bysame, \emph{Characters of reductive groups over a finite field}, Annals of
  mathematics studies, no. 107, Princeton {U}niversity {P}ress, 1984.

\bibitem[Lus85]{lu:char1}
\bysame, \emph{Character sheaves {I}}, Adv. Math. \textbf{56} (1985), 193--237.

\bibitem[Lus04]{lu:inductionthm}
\bysame, \emph{An induction theorem for {S}pringer's representations},
  Representation theory of algebraic groups and quantum groups, Adv. {S}tud.
  {P}ure {M}ath., no.~40, 2004, pp.~253--259.

\bibitem[Lus10]{lu:certain}
\bysame, \emph{On certain varieties attached to a {W}eyl group element},
  Available at \url{http://arxiv.org/abs/1012.2074}, 2010, ar{X}iv:1012.2074
  [math.{R}{T}].

\bibitem[Lus11]{lu:weyltounip}
\bysame, \emph{From conjugacy classes in the {W}eyl group to unipotent
  classes}, Representation {T}heory({A}n {E}lectronic {J}ournal of the {AMS})
  \textbf{15} (2011), 494--530.

\bibitem[Lus12]{lu:homogeneity}
\bysame, \emph{Elliptic elements in a {W}eyl group: a homogeneity property},
  Representation {T}heory({A}n {E}lectronic {J}ournal of the {AMS}) \textbf{16}
  (2012), 127--151.

\bibitem[Lus14]{lu:distinguished}
\bysame, \emph{Distinguished conjugacy classes and elliptic {W}eyl group
  elements.}, Representation {T}heory \textbf{18} (2014), no.~8, 223 -- 277.

\bibitem[Mac95]{macdonald}
Macdonald, I.~G., \emph{Symmetric functions and {H}all polynomials}, second
  ed., Oxford mathematical monographs, Oxford {U}niversity {P}ress, 1995.

\bibitem[Sta86]{stanley}
Stanley, R.~P., \emph{Enumerative combinatorics}, vol.~2, Cambridge studies in
  advanced mathematics, no.~62, Cambridge {U}niversity {P}ress, 1986.

\end{thebibliography}

\end{document}